\documentclass{amsart}

\usepackage{graphics}
\usepackage{epsfig}
\usepackage{amsmath}
\usepackage{amssymb,fancyhdr,txfonts,pxfonts}

\usepackage{xcolor}
\usepackage{color, colortbl}
\usepackage{graphicx}
\usepackage{marvosym}

\usepackage{paralist}
\usepackage{multicol}

\usepackage{pstricks,pst-text,pst-grad,pst-node,pst-3dplot,pstricks-add,pst-poly,pst-coil} 
\usepackage{pst-fun,pst-blur} 

\usepackage{hyperref}

\newtheorem{theorem}{Theorem}
\theoremstyle{plain}

\newtheorem{corollary}{Corollary}

\newtheorem{definition}{Definition}
\newtheorem{example}{Example}

\newtheorem{lemma}{Lemma}
\newtheorem{notation}{Notation}

\newtheorem{proposition}{Proposition}
\newtheorem{properties}{Properties}
\newtheorem{remark}{Remark}

\newcommand{\cyc}{\mbox{cyc}}

\begin{document}

\title[Affine Plane Ratio of Linear Points]{Advances in the Geometry of the Ratio of Linear Points in the Desargues Affine Plane Skew Field}

\author[Orgest ZAKA]{Orgest ZAKA}
\address{Orgest ZAKA: Department of Mathematics-Informatics, Faculty of Economy and Agribusiness, Agricultural University of Tirana, Tirana, Albania}
\email{ozaka@ubt.edu.al, gertizaka@yahoo.com, ozaka@risat.org}

\author[James F. Peters]{James F. Peters}
\address{James F. PETERS: Department of Electrical \& Computer Engineering, University of Manitoba, WPG, MB, R3T 5V6, Canada and Department of Mathematics, Faculty of Arts and Sciences, Ad\.{i}yaman University, 02040 Ad\.{i}yaman, Turkey}
\thanks{The research has been supported by the Natural Sciences \& Engineering Research Council of Canada (NSERC) discovery grant 185986, Instituto Nazionale di Alta Matematica (INdAM) Francesco Severi, Gruppo Nazionale  per le Strutture Algebriche, Geometriche e Loro Applicazioni grant 9 920160 000362, n.prot U 2016/000036 and Scientific and Technological Research Council of Turkey (T\"{U}B\.{I}TAK) Scientific Human Resources Development (BIDEB) under grant no: 2221-1059B211301223.}
\email{James.Peters3@umanitoba.ca}

\dedicatory{Dedicated to Girard Desargues and R.V. Banavar}

\subjclass[2010]{51A30; 51E15, 51N25, 30C20, 30F40}

\begin{abstract}
This paper introduces advances in the geometry of the ratio of either two or three points in a line in the Desargues affine plane, and we see this as a ratio of elements of skew field which are constructed over a line in Desargues affine plane. The results given here have a clean, geometric presentation based Desargues affine plan axiomatics and definitions of addition and multiplication of points on a line in this plane, and for skew field properties.  The results in this paper are: (1) study of properties for ratio of two and three points, in a line on Desargues affine plane. Also, we discuss the cases related to the "line-skew field" characteristic, when it is two and when it is different from two. (2) we have construct the maps for ratio points-set, for two and three points, and have prove that, this maps are bijections of the lines. (3) set of ratio points (for two and for three points) with addition and multiplication of points, forms a skew fields, for more, this skew fields are sub-skew fields of the 'line-skew field' on Desargues affine plane. (4) Every Dyck polygon containing co-linear ratio vertices in the Desargues affine plane has a free group presentation.
\end{abstract}

\keywords{Co-linear, Ratio, Skew-Field, Desargues Affine Plane}

\maketitle

\section{Introduction}

The foundations for the study of the connections between axiomatic geometry and algebraic structures were set forth by D. Hilbert \cite{Hilbert1959geometry}. And some classic for this are,  E. Artin \cite{Artin1957GeometricAlgebra}, D.R. Huges and F.C. Piper ~\cite{HugesPiper}, H. S. M Coxeter ~\cite{CoxterIG1969}. 
Marcel Berger in \cite{Berger2009geometry12}, Robin Hartshorne
 in \cite{Hartshorne1967Foundations}, etc. Even earlier, in we works \cite{ZakaDilauto, ZakaFilipi2016, FilipiZakaJusufi, ZakaCollineations, ZakaVertex, ZakaThesisPhd, ZakaPetersIso, ZakaPetersOrder, ZakaMohammedSF, ZakaMohammedEndo} we have brought up quite a few interesting facts about the association of algebraic structures with affine planes and with ’Desargues affine planes’, and vice versa.

In this article we study 'ratio of 2-points' and 'ratio of 3-points' in a line of the Desargues affine plane. Earlier, we have shown that on each line on Desargues affine plane, we can construct a skew-field simply and constructively, using simple elements of elementary geometry, and only the basic axioms of Desargues affine plane (see \cite{ZakaFilipi2016}, \cite{FilipiZakaJusufi}, \cite{ZakaThesisPhd}, \cite{ZakaPetersIso} ). 
In this paper, we consider dilations and translations entirely in the Desargues affine plane (see \cite{ZakaDilauto}, \cite{ZakaCollineations}, \cite{ZakaThesisPhd}, \cite{ZakaPetersIso}).

The novelty in this paper is that we achieve our results without the use coordinates.  Instead, we see the points of a line as elements of the skew field which is constructed over this line and we make use of properties enjoyed by transformations in Desargues affine planes such as parallel projection, translations and dilations and is proved that the this-transformations preserve the ratio of 2 and 3 points in a line.  In addition, we identify and present every Dyck polygons containing colinear points in the Desargues affine plane as a free group, which make it possible to achieve a concise view of complex geometric constructions.

\section{Preliminaries}
This secion gives a brief introduction to addition and multiplication of co-linear points in the Desargues affine plane and algebraic properties of skew fields.

\subsection{Desargues Affine Plane}$\mbox{}$\\

Let $\mathcal{P}$ be a nonempty space, $\mathcal{L}$ a nonempty subset of $\mathcal{P}$. The elements $p$ of $\mathcal{P}$ are points and an element $\ell$ of $\mathcal{L}$ is a line. 

\begin{definition}
The incidence structure $\mathcal{A}=(\mathcal{P}, \mathcal{L},\mathcal{I})$, called affine plane, where satisfies the above axioms:

\begin{compactenum}[1$^o$]
\item For each points $\left\{P,Q\right\}\in \mathcal{P}$, there is exactly one line $\ell\in \mathcal{L}$ such that $\left\{P,Q\right\}\in \ell$.

\item For each point $P\in \mathcal{P}, \ell\in \mathcal{L}, P \not\in \ell$, there is exactly one line $\ell'\in \mathcal{L}$ such that
$P\in \ell'$ and $\ell\cap \ell' = \emptyset$\ (Playfair Parallel Axiom~\cite{Pickert1973PlayfairAxiom}).   Put another way,
if the point $P\not\in \ell$, then there is a unique line $\ell'$ on $P$ missing $\ell$~\cite{Prazmowska2004DemoMathDesparguesAxiom}.

\item There is a 3-subset of points $\left\{P,Q,R\right\}\in \mathcal{P}$, which is not a subset of any $\ell$ in the plane.   Put another way,
there exist three non-collinear points $\mathcal{P}$~\cite{Prazmowska2004DemoMathDesparguesAxiom}.
\end{compactenum}
\end{definition}

\emph{\bf Desargues' Axiom, circa 1630}~\cite[\S 3.9, pp. 60-61] {Kryftis2015thesis}~\cite{Szmielew1981DesarguesAxiom}.   Let $A,B,C,A',B',C'\in \mathcal{P}$ and let pairwise distinct lines  $\ell^{AA_1} , \ell^{BB'}, \ell^{CC'}, \ell^{AC}, \ell^{A'C'}\in \mathcal{L}$ such that
\begin{align*}
\ell^{AA_1} \parallel \ell^{BB'} \parallel \ell^{CC'} \ \mbox{(Fig.~\ref{fig:DesarguesAxiom}(a))} &\ \mbox{\textbf{or}}\
\ell^{AA_1} \cap \ell^{BB'} \cap \ell^{CC'}=P.
 \mbox{(Fig.~\ref{fig:DesarguesAxiom}(b) )}\\
 \mbox{and}\  \ell^{AB}\parallel \ell^{A'B'}\ &\ \mbox{and}\ \ell^{BC}\parallel \ell^{B'C'}.\\
A,B\in \ell^{AB}, A'B'\in \ell^{A'B'},  &\ \mbox{and}\ B,C\in \ell^{BC},  B'C'\in \ell^{B'C'}.\\
A\neq C, A'\neq C', &\ \mbox{and}\ \ell^{AB}\neq \ell^{A'B'}, \ell^{BC}\neq \ell^{B'C'}.
\end{align*}

\begin{figure}[htbp]
	\centering
		\includegraphics[width=0.85\textwidth]{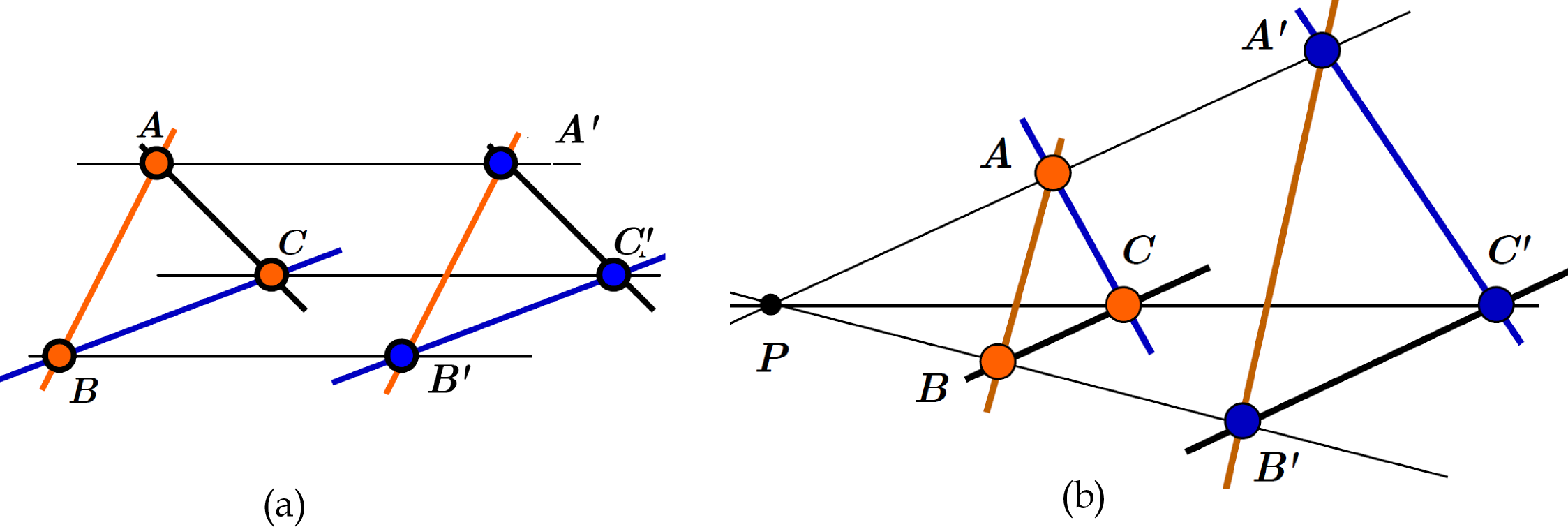}
	\caption{Desargues Axioms: (a) For parallel lines $\ell^{AA'} \parallel \ell^{BB'} \parallel \ell^{CC'}$; (b) For lines which are cutting in a single point $P$,  $\ell^{AA'} \cap \ell^{BB'} \cap \ell^{CC'}=P$.}
	\label{fig:DesarguesAxiom}
\end{figure}

Then $\boldsymbol{\ell^{AC}\parallel \ell^{A'C'}}$.   \qquad \textcolor{blue}{$\blacksquare$}

\begin{example}
In Euclidean plane, three vertexes $ABC$ and $A'B'C'$, are similar (in (a) are equivalent-triangle and in (b) are homothetical-triangle) the parallel lines,  $\ell^{AC}, \ell^{A'C'}\in \mathcal{L}$ in Desargues' Axiom are represented in Fig.~\ref{fig:DesarguesAxiom}.  In other words, the side $AC$ of the triangle of $\bigtriangleup ABC$ is parallel with the side $A'C'$ of the triangle $\bigtriangleup A'B'C'$, provided the restrictions on the points and lines in Desargues' Axiom are satisfied.
\qquad \textcolor{blue}{$\blacksquare$}
\end{example}

\noindent A {\bf Desargues affine plane} is an affine plane that satisfies Desargues' Axiom.

\vspace*{0.3cm}
 
\begin{notation}
Three vertexes $ABC$ and $A'B'C'$, which, fulfilling the conditions of the Desargues Axiom, we call \emph{'Desarguesian'}.
\end{notation}

\subsection{Addition and Multiplication of points in a line of Desargues affine plane}$\mbox{}$\\

\textbf{Addition of points in a line in affine plane}. In an Desargues affine plane $\mathcal{A_D}=(\mathcal{P},\mathcal{L},\mathcal{I})$ we fix two different points $O,I\in \mathcal{P},$ which, according to Axiom 1, determine a line $\ell^{OI}\in \mathcal{L}.$ Let $A$ and $B$ be two arbitrary points of a line $\ell^{OI}$. In plane $\mathcal{A_D}$ we choose a point $B_{1}$ not incident with $\ell^{OI}$: $B_{1}\notin \ell^{OI}$ (we call the auxiliary point). Construct line $\ell_{OI}^{B_{1}},$ which is only according to the Axiom 2. Then construct line $\ell_{OB_{1}}^{A},$ which also is the only according to the Axiom 2. Marking their intersection $P_{1}=\ell_{OI}^{B_{1}}\cap \ell_{OB_{1}}^{A}.$ Finally construct line $\ell_{BB_{1}}^{P_{1}}.$ For as much as $\ell^{BB_{1}}$ cuts the line $\ell^{OI}$ in point $B$, then this line, parallel with $\ell^{BB_{1}}$, cuts the line $\ell^{OI}$ in a single point $C$, this point we called the addition of points $A$ with point $B$ (Figure \ref{fig:FigureAdMult} (a)).

\textbf{Multiplication of points in a line in affine plane}.
Choose in the plane $\mathcal{A_D}$ one point $B_{1}$ not incident with lines $\ell^{OI},$ and construct the line $\ell^{IB_{1}}$. Construct the line $\ell_{IB_{1}}^{A},$ which is the only accoding to the Axiom 2 and cutting the line $\ell^{OB_{1}}$. Marking their intersection with $P_{1}=\ell
_{IB_{1}}^{A}\cap OB_{1}.$ Finally, construct the line $\ell
_{BB_{1}}^{P_{1}}.$ Since $\ell^{BB_{1}}$ cuts the line $\ell^{OI}$ in a single point $B$, then this line, parallel with $\ell^{BB_{1}}$, cuts the line $\ell^{OI}$ in one single point $C$, this point we called the multiplication of points $A$ with point $B$ (Figure \ref{fig:FigureAdMult} (b)).

\begin{figure}[htbp]
\centering%
\includegraphics[width=0.92\textwidth]{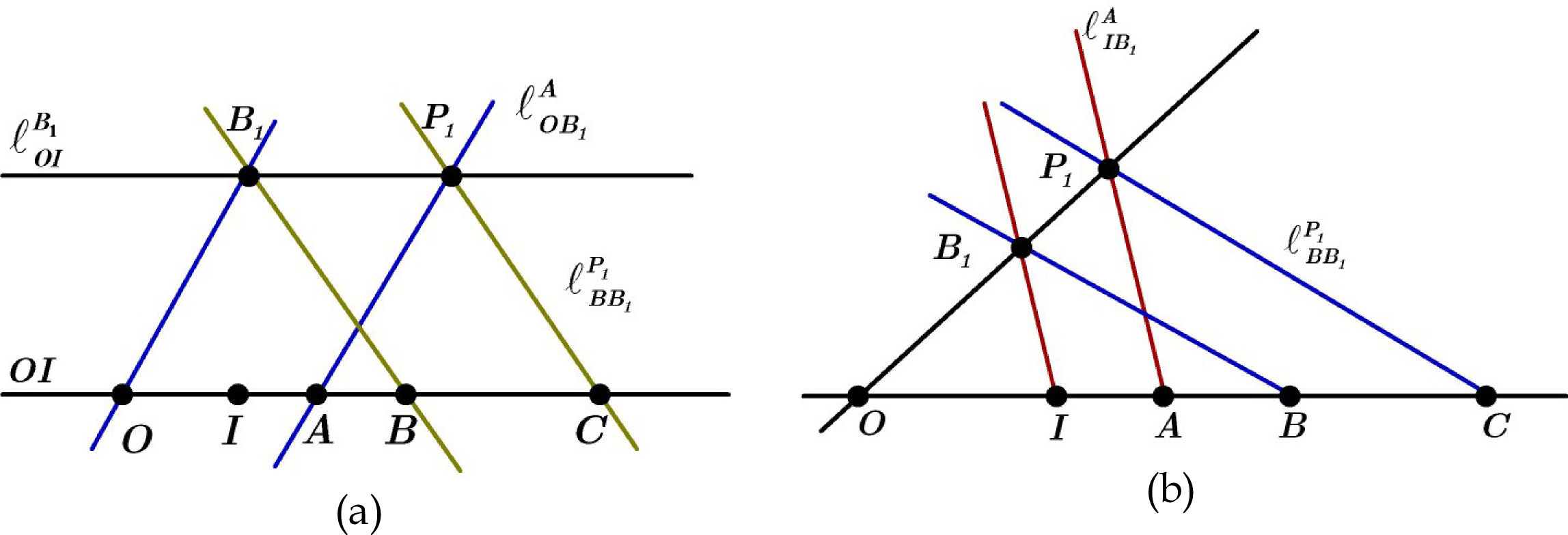}
\caption{ (a) Addition of points in a line in affine plane, 
(b) Multiplication of points in a line in affine plane}
\label{fig:FigureAdMult}
\end{figure}

The process of construct the points $C$ for adition and multiplication of points in $\ell^{OI}-$line in affine plane, is presented in the tow algorithm form  

\begin{multicols}{2}
\textsc{Addition Algorithm}
\begin{description}
	\item[Step.1] $B_{1}\notin \ell^{OI}$
	\item[Step.2] $\ell_{OI}^{B_{1}}\cap \ell_{OB_{1}}^{A}=P_{1}$
	\item[Step.3] $\ell_{BB_{1}}^{P_{1}}\cap \ell^{OI}=C(=A+B)$
\end{description}

\textsc{Multiplication Algorithm}
\begin{description}
	\item[Step.1] $B_{1}\notin \ell^{OI}$
	\item[Step.2] $\ell_{IB_{1}}^{A}\cap \ell^{OB_{1}}=P_{1}$
	\item[Step.3] $\ell_{BB_{1}}^{P_{1}}\cap \ell^{OI}=C(=A\cdot B)$
\end{description}
\end{multicols}

In \cite{ZakaThesisPhd} and \cite{FilipiZakaJusufi}, we have prove that $(\ell^{OI}, +, \cdot)$ is a skew field in Desargues affine plane, and is field (commutative skew field) in the Papus affine plane.

\subsection{Some algebraic properties of Skew Fields} $\mbox{}$\\

I n this section $K$ will denote a skew field~\cite{Herstein1968NR} and $z[K]$ its center, where is the set $K$ such that
\[
z[K]=\left\{k \in K \quad |\quad ak=ka, \quad \forall a \in K \right\}
\]
\begin{proposition}
$z[K]$ is a commutative subfield of a skew field $K$.
\end{proposition}
\begin{proof}
Show first, that $O, I \in z[K]$, because $O\cdot X=X\cdot O$ and $I\cdot X=X\cdot I$, $\forall X\in K$ (for all point $X\in \ell^{OI}-$line), so $z[K]$ is non-empty.

If $A,B \in z[K]$ and $C\in K$, have
\[ (A-B)C=AC-BC=CA-CB=C(A-B), \]
then 
\[\forall A,B \in z[K] \Rightarrow A-B \in z[K]. \]
let's see now, 
\[ (AB)C=A(BC)=A(CB)=(AC)B=(CA)B=C(AB),\]
then, we have  
\[\forall A,B \in z[K] \Rightarrow AB \in z[K]. \]
If $A \in z[K]$, and $A \neq O$, then $A^{-1} \in z[K]$, truly,
is clear that $(A^{-1})^{-1}=A$, and let $B\in K$, then
\[B^{-1}A^{-1}=(AB)^{-1}=(BA)^{-1}=A^{-1}B^{-1}, \]
thus, for $A \in z[K]$, and $A \neq O$, and $\forall B\in z[K]$, we have that,
\[B^{-1}A^{-1}=A^{-1}B^{-1} \Rightarrow  A^{-1} \in z[K] \]
Thus, $z[K]$ is a commutative subring of $K$, then is a subfield of skew field $K$.
\end{proof}
Let now $p\in K$ be a fixed element of the skew field $K$. We will denote $z_K(p)$ the centralizer in $K$ of the element $p$, where is the set,
\[
z_K(p)=\left\{k \in K | pk=kp, \right\}.
\]
$z_K(p)$ is sub skew field of K, but, in general, it is not commutative.

Let $K$ be a skew field, $p\in K$, and let us denote by $[p_K]$ the conjugacy class of $p$:
\[
[p_K]= \left\{q^{-1}pq \quad|\quad q \in K \setminus \{0\} \right\}
\]
If, $p\in z[K]$, for all $q \in K$ we have that $q^{-1}pq=p$

\section{Ratio of two and three Points in a line of Desargues affine plane}

\bigskip
In this section, we will give the main result in this article.
We will study the 'ration of two points' and 'ratio of three points' on a line in Desargues affine planes, we utilize a method that is naive and direct, in the abstract and axiomatic sense of it, without requiring the concept of coordinates, and metrics. We will defined the 'ratio of two points' and 'ratio of three points' as a point in the Desargues affine plane.  For achieve our results, we do not use coordinates, and follow simple techniques, of elementary geometry, for proofs.  We are mainly based on our previous work, on the close connection we found between the lines, in Desargues affine planes and skew fields. 

Mainly, we rely on our results regarding the addition and multiplication of points on a line of the Desargues affine plane, and the fact that a line (set of points), with addition and multiplication, forms a skew field (for more, see \cite{ZakaThesisPhd}, \cite{ZakaFilipi2016}, \cite{FilipiZakaJusufi}, \cite{ZakaVertex}, \cite{ZakaCollineations}, \cite{ZakaDilauto},  \cite{ZakaPetersIso}, \cite{ZakaPetersOrder}).

We will select some transformations (some of them we are defining ourselves) which are invariant regarding the ratio of points. 
As well on the properties enjoyed by transformations in Desargues affine planes, like, parallel projection, translations and dilations (see \cite{ZakaCollineations}, \cite{ZakaDilauto},  \cite{ZakaPetersIso}), we prove that this transformations preserving 'ratio of two points' and 'ratio of three points' in a line on a Desargues affine plane. We use skew fields properties, for our proofs.

\subsection{Ratio of 2-Points in a line of Desargues affine plane}$\mbox{}$\\

In this section we will study the 'ration' of 2-points on a line in Desargues affine planes. We will define 'ratio' as a point in the Desargues affine plane. 

\begin{definition} \label{ratio2points}
Lets have two different points $A,B \in \ell^{OI}-$line, and $B\neq O$, in Desargues affine plane. We define as ratio of this tow points, a point $R\in \ell^{OI}$, such that,
\[R=B^{-1}A
\] 
we mark this, with,
\[
R=r(A:B)=B^{-1}A
\]
\end{definition}

For a 'ratio-point' $R \in \ell^{OI}$, and for point $B\neq O$ in line $\ell^{OI}$, is a unique defined point, $A \in \ell^{OI}$, such that $R=B^{-1}A=r(A:B)$.
 
\begin{figure}[htbp]
	\centering
		\includegraphics[width=0.8\textwidth]{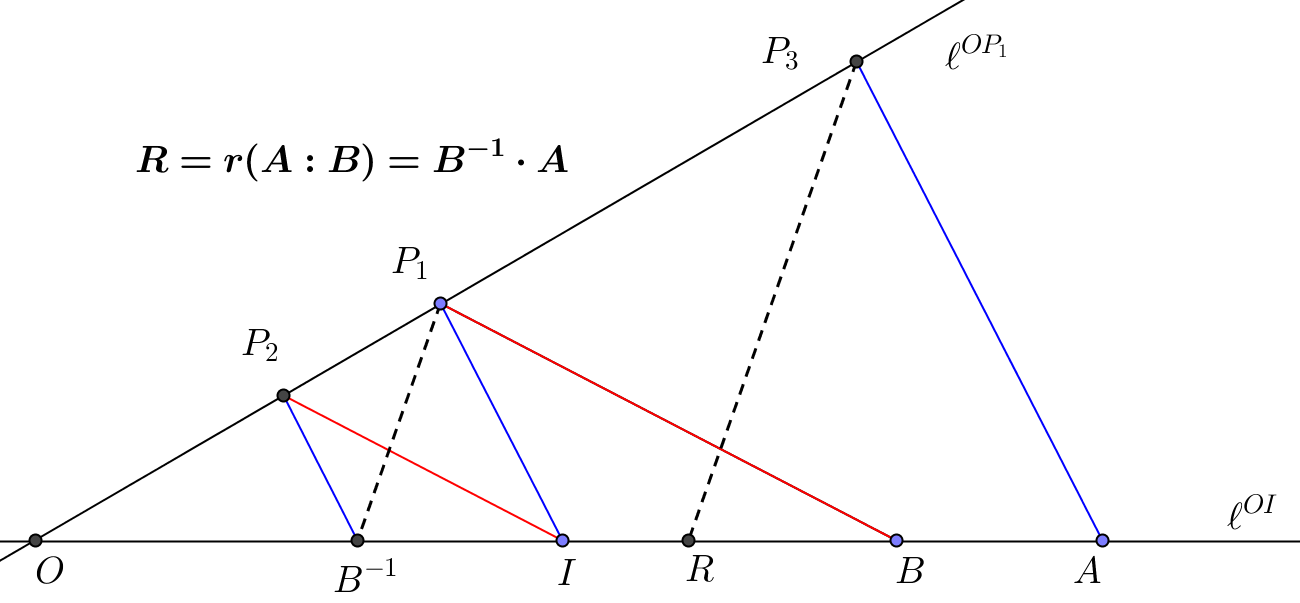}
		\caption{Ilustrate the Ratio-Point, of 2-Points in a line of Desargues affine plane $R=r(A:B)=B^{-1}A$.}
		\label{Ratio2points}
\end{figure}

\begin{remark}
Our caution is maximum, since the points of a line in Desargues affine planes are also elements of skew-fields built on it, and the commutative property for multiplication does not hold! For this reason we have
\[ B^{-1}\cdot A \neq A\cdot B^{-1}.\]
However, if we had defined the ratio of 2-points as,
\[r(A:B)=A\cdot B^{-1}, \]
the results would be similar. However, we are keeping our definition!
\end{remark}

\begin{properties}\label{properties:2points}
From Def.\ref{ratio2points}, we derive a number of properties related to the position of the points $A,B$ in $\ell^{OI}-$line in the Desargues affine plane.

\begin{itemize}
  \item If $A=B\neq O$, then $r(A:B)=r(A:A)=A^{-1}\cdot A=I$.
	\item If $A=O$ and $B\neq O$, then $r(A:B)=r(O:B)=B^{-1}\cdot O=O$.
	\item If $A=I$ and $B\neq O$, then $r(A:B)=r(O:I)=B^{-1}\cdot I=I$.
	\item If $A=I$ and $B\neq I$, then $r(A:B)=r(I:I)=I^{-1}\cdot I=I$.
	\item If $A\neq O$ and $B=O$, then $r(A:B)=r(A:O)=O^{-1}\cdot A=\infty$ (infinite point).
\end{itemize}
\end{properties}

\begin{theorem}
If have two different points $A,B \in \ell^{OI}-$line, and $B\neq O$, in Desargues affine plane, then,
\[
r^{-1}(A:B)=r(B:A)
\]
\end{theorem}
\proof
From definition \ref{ratio2points}, we have that $r(A:B)=B^{-1}A$, we appreciate,
\[
\begin{aligned}
r^{-1}(A:B)&=[r(B:A)]^{-1}\\
	&= [B^{-1}A]^{-1}\\
	&= A^{-1}[B^{-1}]^{-1}	\\
	&= A^{-1}B \quad \text{(in skew fields $(ab)^{-1}=b^{-1}a^{-1}, \forall a,b$)} \\
&=r(B:A)
\end{aligned}
\]
\qed

\begin{theorem}
For three collinear point $A,B,C$ and $C\neq O$, in $\ell^{OI}-$line, have,
\[
r(A+B:C)=r(A:C)+r(B:C)
\]
\end{theorem}
\proof

From definition \ref{ratio2points}, we have that $r(A:B)=B^{-1}A$, we appreciate,
\[
\begin{aligned}
r(A+B:C)&=C^{-1}(A+B)\\
	&= C^{-1}A+C^{-1}B) \\
	&\text{(from distribution property of skew fields $K=(\ell^{OI},+,\cdot)$)}\\
	&= r(A:C)+r(B:C)
\end{aligned}
\]
\qed

\begin{theorem}
For three collinear point $A,B,C$ and $C\neq O$, in $\ell^{OI}-$line, have,
\begin{enumerate}
	\item $r(A\cdot B:C)=r(A:C)\cdot B.$
	\item $r(A:B\cdot C)=C^{-1}r(A:C).$
\end{enumerate}
\[
r(A\cdot B:C)=r(A:C)\cdot B
\]
\end{theorem}
\proof
(1) From definition \ref{ratio2points}, we have that $r(A:B)=B^{-1}A$, we appreciate,
\[
\begin{aligned}
r(A\cdot B:C)&=C^{-1}(A\cdot B)\\
	&=(C^{-1}A)\cdot B  \\
	&\text{(from associative property of skew fields $K=(\ell^{OI},+,\cdot)$)}\\
	&= r(A:C)\cdot B.
\end{aligned}
\]
And for (2),
\[
\begin{aligned}
r(A:B\cdot C)&=[BC]^{-1}A\\
	&=(C^{-1}B^{-1})\cdot A  \\
	&=C^{-1}(B^{-1}\cdot A)   \\
	&\text{(from associative property of skew fields $K=(\ell^{OI},+,\cdot)$)}\\
	&=C^{-1}r(A:C). 
\end{aligned}
\]
\qed

\begin{theorem}
Let's have the points $A,B \in \ell^{OI}-$line where $B\neq O$.  Then have that,
\[
r(A:B)=r(B:A) \Leftrightarrow A=B.
\]
\end{theorem}
\proof
Suppose that $r(A:B)=r(B:A)$, so $B^{-1}A=A^{-1}B$, but $B^{-1}A=[A^{-1}B]^{-1}$, then we have,
\[
[A^{-1}B]^{-1}=A^{-1}B.
\]
But points, $A,B, A^{-1}, B^{-1}, A^{-1}B, B^{-1}A \in \ell^{OI}$, and $(\ell^{OI}, +, \cdot)$ is skew field. From skew field properties: If $a^{-1}=a \Rightarrow a^2=1$, then element $a$ is \emph{idempotent} in skew field, but, but, the idempotent elements in the skew field are alone $0$ and $1$ ({see, {\em e.g.},}~\cite{Herstein1968NR}). So to us, we have that
\[
A^{-1}B-\text{Idempotent element of skew field $(\ell^{OI}, +, \cdot)$,}
\]
hence,
\[
A^{-1}B=O\quad \text{or}\quad A^{-1}B=I
\]
so,
\[
A^{-1}B=O \Rightarrow B=O \quad \text{which contradicts}
\]
then, we have that,
\[
 A^{-1}B=I  \Rightarrow B=A.
\]
\qed

We can say, now, that all two points in $\ell^{OI}-$line produce a point in it, which we have called 'ratio of two points'. So, we can construct the map:
\[
r_{B}(X): \ell^{OI} \to \ell^{OI},
\]
such that,
\[
r_{B}(X)=r(X:B)=B^{-1}X, \forall X \in \ell^{OI}
\]

\begin{theorem}
This ratio-map, $r_{B}: \ell^{OI} \to \ell^{OI}$ is a bijection in $\ell^{OI}-$line in Desargues affine plane. 
\end{theorem}
\proof
\textbf{Injective:} If have,
\[r_{B}(X)=r_{B}(Y)\Leftrightarrow B^{-1}X=B^{-1}Y
\]
then
\[
B^{-1}X-B^{-1}Y=O \Leftrightarrow B^{-1}(X-Y)=O.
\]

Let's remember that these are the points of the skew fields $K$, and the skew field has no 'divisor of zero', for this reason we have,
\[
B^{-1}(X-Y)=O\Leftrightarrow B^{-1}=O\quad \text{or}\quad X-Y=O,
\]
but $B^{-1}\neq O$, then have that $X-Y=O \Leftrightarrow X=Y.$

\textbf{Surjective:} $\forall R\in \ell^{OI}, \exists X \in \ell^{OI}$, such that $r_B(X)=R$. \\
We know that, $r_B(X)=B^{-1}X=R\Leftrightarrow X=B\cdot R$, so,
\[
\forall R\in \ell^{OI}, \exists X=B\cdot R \in \ell^{OI}, \quad\text{i.e.,}\quad r_B(X)=r_B(B\cdot R)=B^{-1}(BR)=(B^{-1}B)R=R.
\]
\qed

For this map we have, $r_{B}(B)=r(B:B)=B^{-1}B=I$, $r_{B}(O)=r(O:B)=B^{-1}O=I$ and  $r_{B}(I)=r(I:B)=B^{-1}I=B^{-1}$, and for all point $X \in \ell^{OI}$ have $r_{B}(X)=r(X:B)$, which is the ratio of two points $X$ and $B$.

The map $r_{B}(X)=r(X:B)$ is a point in $\ell^{OI}$ line, for this reason, addition of points (our algorithm for collecting points on a line of Desargues affine plane), which we used earlier, is valid. We see that the set,
\[
\mathcal{R}_2=\{r_{B}(X)|\forall X\in \ell^{OI}, \}
\]
It's clear that $\mathcal{R}_2 \subseteq \ell^{OI},$ and $\mathcal{R}_2 \neq \emptyset$, below we will prove that this set, is a sub-skew field of $K=(\ell^{OI},+, \cdot )$.  Let's see now the following theorems,
\begin{theorem}\label{thm.01}
Let's have the point $B$ different for point $O$, in  $\ell^{OI}-$line. The ratio-maps-set $\mathcal{R}_2=\{r_{B}(X)|\forall X\in \ell^{OI} \}$ forms a \emph{Abelian} (commutative) group with addition of points in $\ell^{OI}-$line.
\end{theorem}
\proof
\textbf{1.} (Associativity) $\forall X,Y,Z \in \ell^{OI}$, we have that:
\[
\begin{aligned}
\left[r_{B}(X)+r_{B}(Y)\right]+r_{B}(Z)&=[B^{-1}X+B^{-1}Y]+B^{-1}Z \\
&= B^{-1}[X+Y]+B^{-1}Z \\
&= B^{-1}[(X+Y)+Z] \\
&\text{(from asociatiove properties of addition groups $(\ell^{OI},+)$.)}\\
&= B^{-1}[X+(Y+Z)] \\
&=B^{-1}X+B^{-1}(X+Y) \\
&=B^{-1}X+[B^{-1}Y+B^{-1}Z]\\
&=r_{B}(X)+[r_{B}(Y)+r_{B}(Z)],
\end{aligned}
\]
so,
\[
[r_{B}(X)+r_{B}(Y)]+r_{B}(Z)=r_{B}(X)+[r_{B}(Y)+r_{B}(Z)], \forall X,Y,Z \in \ell^{OI}.
\]
\textbf{2.} (Commuativity) $\forall X,Y \in \ell^{OI}$, we have that:
\[
\begin{aligned}
r_{B}(X)+r_{B}(Y)&=B^{-1}X+B^{-1}Y\\
&= B^{-1}[X+Y] \\
&\text{(from commutative property of addition groups $(\ell^{OI},+)$.)}\\
&= B^{-1}[Y+X] \\
&=B^{-1}Y+B^{-1}X \\
&=r_{B}(Y)+r_{B}(X),
\end{aligned}
\]
so
\[  r_{B}(X)+r_{B}(Y)=r_{B}(Y)+r_{B}(X),  \forall X,Y \in \ell^{OI}\]

\textbf{3.} ('zero') $\exists Z \in \ell^{OI}, \forall X \in \ell^{OI}$, we have that:
\[
\begin{aligned}
r_{B}(X)+r_{B}(Z)&=r_{B}(X)\\
B^{-1}X+B^{-1}Z &=B^{-1}X\\
 B^{-1}[X+Z] &=B^{-1}X
\end{aligned}
\]
so
\[ B^{-1}[X+Z]=B^{-1}(X)\Rightarrow X+Z=X \Leftrightarrow Z=O. \]
Hence the 'zero' element is $r_{B}(O)$.

\textbf{4.} ('opposite') $\forall X \in \ell^{OI}, \exists Y \in \ell^{OI}$, we have that:
\[
\begin{aligned}
r_{B}(X)+r_{B}(Y)&=r_{B}(O)\\
B^{-1}X+B^{-1}Y&=B^{-1}O\\
 B^{-1}[X+Y]&=O 
\end{aligned}
\]
so
\[ B^{-1}[X+Y]=B^{-1}O\Rightarrow X+Y=O \Leftrightarrow Y=-X. \]
Hence the 'opposite' element of $r_{B}(X)$ is $r_{B}(-X)$, $\forall X \in \ell^{OI}$.
\qed

\begin{theorem} \label{thm.02}
Let's have the point $B$ different for point $O$, in  $\ell^{OI}-$line. The ratio-maps-set $\{r_{B}(X)|\forall X\in \ell^{OI} \}$ forms a group with 'multiplication of points in $\ell^{OI}-$line.
\end{theorem}
\proof
\textbf{1.} (Associativity) $\forall X,Y,Z \in \ell^{OI}$, we have that:
\[
[r_{B}(X) \cdot r_{B}(Y)]\cdot r_{B}(Z)=[B^{-1}X \cdot B^{-1}Y]\cdot B^{-1}Z \]

\textbf{2.} ('unitary') $\exists E \in \ell^{OI}, \forall X \in \ell^{OI}$, we have that:
\[
\begin{aligned}
r_{B}(X)\cdot r_{B}(E)&=r_{B}(X)\\
B^{-1}X \cdot B^{-1}E=B^{-1}X\\
 X \cdot B^{-1}E&=X \\
B^{-1}E&=X^{-1}X\\
B^{-1}E&=I\\
E&=B
\end{aligned}
\]
so unitary-element element is $r_{B}(B)$.

\textbf{3.} ('inverse') $\forall X \in \ell^{OI}, \exists Y \in \ell^{OI}$, we have that:
\[
\begin{aligned}
r_{B}(X)\cdot r_{B}(Y)&=r_{B}(B)\\
B^{-1}X \cdot B^{-1}Y&=I\\
 X \cdot B^{-1}Y&=B\\
B^{-1}Y&=X^{-1}\cdot B\\
Y&=B\cdot X^{-1}\cdot B
\end{aligned}
\]
so
\[ r_{B}(Y)=B^{-1}Y=B^{-1}[B\cdot X^{-1}\cdot B]=X^{-1}\cdot B=r_X(B). \]
Hence the 'inverse' element of $r_{B}(X)$ is $r_{X}(B)$, $\forall X \in \ell^{OI}$.
\qed

\begin{theorem} \label{thm.03}
In the ratio-maps-set $\{r_{B}(X)|\forall X\in \ell^{OI} \}$, for three elements $r_{B}(X)$, $r_{B}(Y)$, $r_{B}(Z)$ of its, have true, that equations
\begin{enumerate}
	\item $r_{B}(X)\cdot [r_{B}(Y)+r_{B}(Z)]=r_{B}(X)\cdot r_{B}(Y) +r_{B}(X)\cdot r_{B}(Z)$, and,
	\item $[r_{B}(X)+r_{B}(Y)]r_{B}(Z)=r_{B}(X)\cdot r_{B}(Z) +r_{B}(Y)\cdot r_{B}(Z)$
\end{enumerate}
\end{theorem}
\proof
Indeed $r_{B}(X)\cdot [r_{B}(Y)+r_{B}(Z)]=B^{-1}X[B^{-1}Y+B^{-1}Z]$, but the 'ratio' points are points in $\ell^{OI}-$line, so they are also points of skew fields $K=(\ell^{OI},+.\cdot )$, therefore, they enjoy the distribution property, therefore
\[B^{-1}X[B^{-1}Y+B^{-1}Z]=(B^{-1}X)(B^{-1}Y)+(B^{-1}X)(B^{-1}Z), \]
so
\[ r_{B}(X)\cdot [r_{B}(Y)+r_{B}(Z)]=r_{B}(X)\cdot r_{B}(Y) +r_{B}(X)\cdot r_{B}(Z).
\]
The other point (2) is proved in the same way.
\qed

From the above three theorems \ref{thm.01}, \ref{thm.02} and \ref{thm.03}, have true, this, 
\begin{theorem}
The ratio-maps-set $\mathcal{R}_2=\{r_{B}(X)|\forall X\in \ell^{OI} \}$, for a fixed point $B$ in $\ell^{OI}-$line, forms a skew-field with 'addition and multiplication' of points in $\ell^{OI}-$line.

This, skew field $(\mathcal{R}_2, +, \cdot)$ is sub-skew field of the skew field $(\ell^{OI}, +, \cdot)$.
\end{theorem}


\subsection{Ratio of 3-points in a line of Desargues affine plane}$\mbox{}$\\

In this section we will study the 'ration' of points on a line in Desargues affine planes. We will define 'ratio' as a point in the Desargues affine plane.

The classical definition of the ratio (which is often called "simple ratio") (see, {\em e.g.},~\cite{Hilbert1959geometry}, \cite{Berger2010geometryRevealed}, \cite{Berger2009geometry12}), which is given as a ratio of lengths, (in coordinate geometry).  So, for example, for three collinear points $A,B,C$,
\[R=\frac{AC}{BC}
\]
where $AC$ is length of segment $[AB]$ and $BC$ is length of segment $[BC]$, or in vectorial form: 'for three collinear points, $A,B,C$, a 'simple ratio' called the 'scalar' $\lambda$, such that,

\[
\vec{AC}=\lambda \vec{BC}
\]

Since we will not use coordinates and metrics, so we have only the axiomatic of Desargues affine plane (our objects are only points and lines), we take as definition the strictly algebraic, somewhat abstractly,

\begin{notation}
For three points $A, B, C$ on a line $\ell^{OI}$ (collinear) in Desargues affine plane, $B-C$ denotes the point $B+(-C)$ and $A-C$ denotes the point $A+(-C)$. The operator '-' denotes the addition of point $B$ with an oposite point of point $C$~ \cite{ZakaFilipi2016}, \cite{FilipiZakaJusufi}, \cite{ZakaThesisPhd}.
\qquad\textcolor{blue}{\Squaresteel}
\end{notation}

\begin{definition}\label{ratiodef}
If $A, B, C$ are three points on a line $\ell^{OI}$ (collinear) in Desargues affine plane, then we define their \textbf{ratio} to be a point $R \in \ell^{OI}$, such that:
\[
(B-C)\cdot R=A-C, \quad \mbox{concisely}\quad R=(B-C)^{-1}(A-C)
\]
and we mark this with
\[r(A,B;C)= (B-C)^{-1}(A-C)
\]
\end{definition}

\begin{figure}[htbp]
	\centering
		\includegraphics[width=0.9\textwidth]{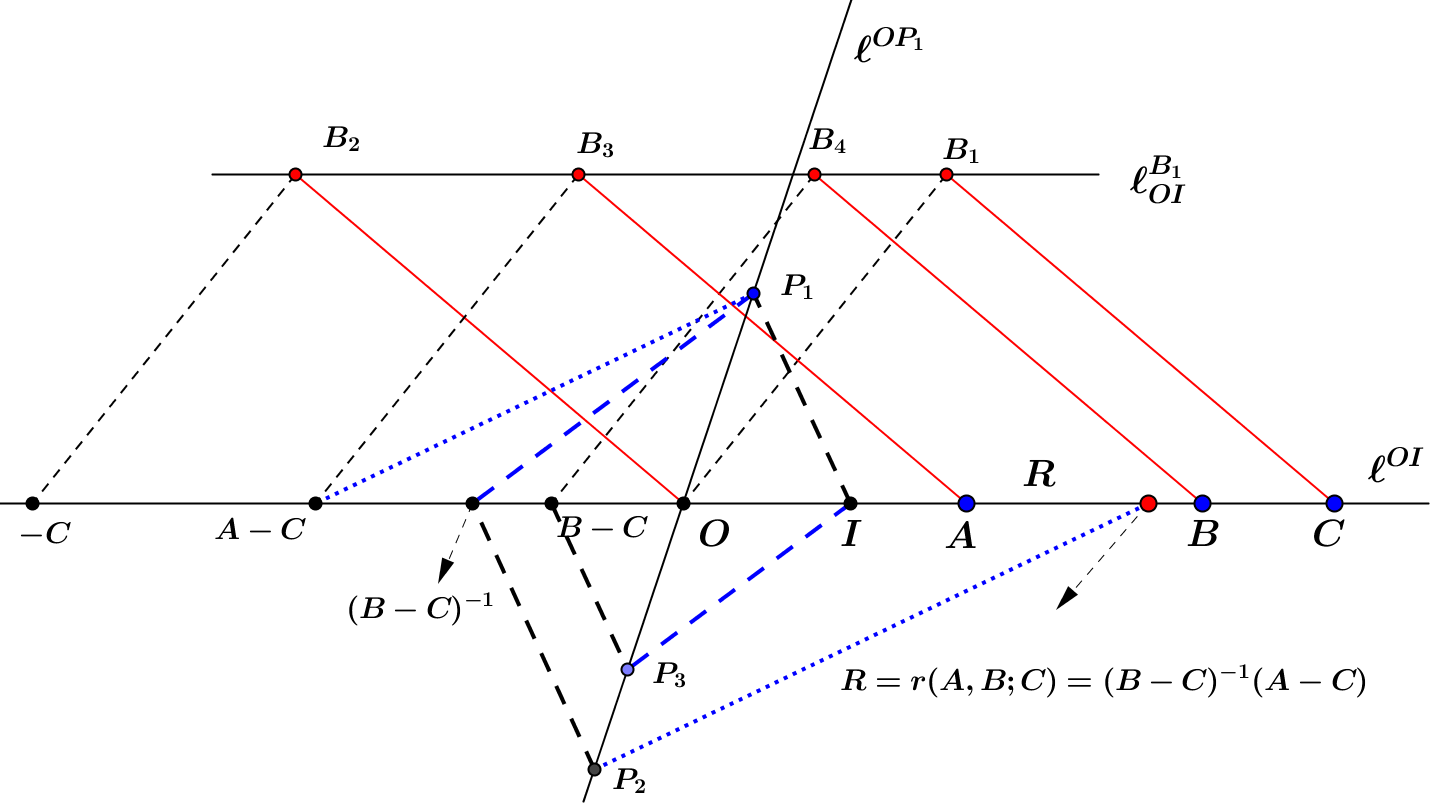}
		\caption{Ratio of 3-Points in a line of Desargues affine plane $R=r(A,B;C)$.}
		\label{ratioThreepoints}
\end{figure}

\begin{example}
In Fig. \ref{Ratio2points}, we have $R$ equal to the ratio of points $A$ and point $B$ in the line $\ell^{OI}$ in the Desargues affine plane.  Similarly, in Fig. \ref{ratioThreepoints}, from Def. \ref{ratiodef}, we have the ratio $R = r(A,B;C)$ of points $A,B,C$.
\qquad\textcolor{blue}{\Squaresteel}
\end{example}

\begin{remark}
Similarly as in the ratio of two points, we can also define a ratio of three points, as 
\[r(A,B;C)=(A-C)(B-C)^{-1},\]
the results would be similar. Except in $\ell^{OI}-$line, in the Desargues affine planes, the usual ratio points differ from ratios in Def.~\ref{ratiodef}, since:
\[ (A-C)(B-C)^{-1} \neq (B-C)^{-1} (A-C).
\]
However, we are keeping our definition!
\qquad\textcolor{blue}{\Squaresteel}
\end{remark}

In the \cite{ZakaThesisPhd}, \cite{ZakaFilipi2016}, \cite{FilipiZakaJusufi},  we prove that an line $\ell^{OI}$in Desargues affine plane, with 'addition' and 'multiplication' of points in it, forms a Skew-Field, so $(\ell^{OI}, +, \cdot )-$is a skew field. For this reason we can define it, 'ratio' also for skew-fields. So, for a skew fields $K$, if we have elements $a, b, c \in K$ then ratio we call the point $r=(b-c)^{-1}(a-c)$.
However, our work is more geometric, so we will follow the geometric interpretation. But we will always keep in mind that: 'a line in Desargues affine plane, is a skew-field'.

For a 'ratio-point' $R\in \ell^{OI}$ and tow points $B,C \in \ell^{OI}$, is a uniquely defined point $A \in \ell^{OI}$, such that $R=r(A,B;C)$.

\begin{theorem}\label{reverse.ratio}
For 3-points $A,B,C$ in a line $\ell^{OI}$ of Desargues affine plane, we have that,
\[r(-A,-B;-C)=r(A,B;C)
\]
\end{theorem}
\proof
For proof of this, we will use the definition of ratio, and   skew-fields properties, so we have,
\begin{equation*}
\begin{aligned}
r(-A,-B;-C)&=(-B-(-C))^{-1}(-A-(-C))\\
	&= (-B+C)^{-1}(-A+C)\\
 &=(-I[B-C])^{-1}[-I](A-C) \\
&=(B-C)^{-1}[-I]^{-1}[-I](A-C) \\
&=(B-C)^{-1}[-I][-I](A-C) \\
&=(B-C)^{-1}(A-C) \\
&=r(A,B;C) 
\end{aligned}
\end{equation*}
From skew fields properties we have that $(ab)^{-1}=b^{-1}a^{-1}$and $ab\neq ba$,  $[-I]^{-1}=-I$, and $[-I][-I]=I.$
\qed

\begin{theorem}\label{inversratio}
For 3-points $A,B,C$ in a line $\ell^{OI}$ in the Desargues affine plane, have
\[r^{-1}(A,B;C)=r(B,A;C)
\]
\end{theorem}
\proof
For proof of this, we will use the definition of ratio, and   skew-fields properties, so we have,
\begin{equation*}
\begin{aligned}
r^{-1}(A,B;C)&=[(B-C)^{-1}(A-C)]^{-1}\\
	&= (A-C)^{-1}[(B-C)^{-1}]^{-1}\\
 &=(A-C)^{-1}(B-C) \\
&=r(B,A;C) 
\end{aligned}
\end{equation*}

\qed

\begin{properties}\label{properties:3points}
From Def.~\ref{ratiodef}, we derive a number of properties related to the position of the points $A,B,C$ in $\ell^{OI}-$line in the Desargues affine plane.
\begin{itemize}
	\item If $A=B=O$ and $C\neq O$, have, 
	\[ r(A,B;C)=(O-C)^{-1}(O-C)=I. \]
	\item If $A=B=O$ and $C=I$, have, 
	\[ r(A,B;C)=(O-I)^{-1}(O-I)=I. \]
	\item If $A=B=I$ and $C=O$, have ,
	\[ r(A,B;C)=(I-O)^{-1}(I-O)=I.\]
	\item If $A=B=I$ and $C\neq I$, have,
	\[ r(A,B;C)=(I-C)^{-1}(I-C)=I. \]
	\item If $A=B=O$ and $C\neq O$ and $C\neq I$, have,
	\[ r(A,B;C)=(O-C)^{-1}(O-C)=I.\]
	\item If $A=B=I$ and $C\neq O$ and $C\neq I$, have,
	\[ r(A,B;C)=(I-C)^{-1}(I-C)=I.\]
	\item If '$A=B\neq O$ and $A=B\neq I$' and $C\neq A=B$, then 
	\[ r(A,B;C)=(A-C)^{-1}(A-C)=I. \]
	\item If $A=C$ and $B \neq A=C$, then $r(A,B;C)=O$.
	\item If $B=C$ and $A \neq C$, then $r(A,B;C)-$is the infinite point, witch we mark with '$\infty$'. \qquad\textcolor{blue}{\Squaresteel}
\end{itemize}
\end{properties}

From Properties~\ref{properties:3points}, we obtain the following result.

\begin{theorem}
For 3-points in $\ell^{OI}-$line in the Desargues affine plane. if $A=B$ and $C\neq A=B$, then $r(A,B;C)=I$, where $I$ is unitary point for 'multiplication of points' in $\ell^{OI}-$line 
\end{theorem}
For three points $A,B,C$ such that $r(A,B;C)=R$, we have defined above, $r(A,B;C)=(B-C)^{-1}(A-C)=R$. We can transform this expression into the form used by L\"{u}ngerburg in \cite{Luneburg1967}, and have that,

$(B-C)^{-1}(A-C)=R\Rightarrow A-C=(B-C)R \Rightarrow A-C=BR-CR$
so
\[A=BR+C-CR \Rightarrow A=BR+C(I-R) \]
This form makes it easier for us to study the following cases, when, $B=C$ and $A=C$! So, 

\begin{itemize}
	\item If $B=C$, then have $A=BR+C(I-R)\Rightarrow A=BR+B(I-R)$, so 
	\[ A=BR+B-BR\Rightarrow \textbf{A=B}.\]
	So 'If $B=C$ then $B=C=A$'.
	\item If $A=C$, then we have $A=BR+C(I-R)\Rightarrow A=BR+A(I-R)$, so,
	\[ A=BR+A-AR\Rightarrow  O=BR-AR \Rightarrow \textbf{B=A}.\]
	So 'If $A=C$ then $A=C=B$'.
\end{itemize}

Immediate from Properties~\ref{properties:3points}, we obtain the following results.

\begin{theorem}
For three points $A,B,C$ in a line $\ell^{OI}$ on Desargues affine plane, i.e., $r(A,B;C)=R$, 
\begin{enumerate}
	\item If $B=C$ then $B=C=A$.
	\item If $A=C$ then $A=C=B$.
\end{enumerate}
\end{theorem}

\begin{theorem} \label{ratio-inverses}
If  $A,B,C$, are three different points, and different from point $O$, in a line $\ell^{OI}$ on Desargues affine plane, then
\[ r(A^{-1},B^{-1};C^{-1})=B[r(A,B;C)]A^{-1}.\] 
\end{theorem}
\proof
Points $A,B,C$ and $A^{-1},B^{-1}, C^{-1}$, are points of $\ell^{OI}-$line in Desargues affine plane, so are and elements of the skew field $K=(\ell^{OI},+,\cdot)$. First we prove that, for tow elements $X,Y$ in a skew field $K$, we have that 
\[X^{-1}-Y^{-1}=Y^{-1}(Y-X)X^{-1}.\]
Indeed
	\[
	\begin{aligned}
	Y^{-1}(Y-X)X^{-1}&=[Y^{-1}(Y-X)]X^{-1}\\
	&=(Y^{-1}Y-Y^{-1}X)X^{-1}\\
	&=(I-Y^{-1}X)X^{-1}\\
	&=IX^{-1}-Y^{-1}(XX^{-1})\\
	&=X^{-1}-Y^{-1}I\\
	&=X^{-1}-Y^{-1}. 
	\end{aligned}
	\]
We use this result in the calculation of $r(A^{-1},B^{-1}; C^{-1})$, we will use the definition of ratio \ref{ratiodef}, and this skew-fields properties, so we have,
\begin{equation*}
\begin{aligned}
r(A^{-1},B^{-1};C^{-1})&=(B^{-1}-C^{-1})^{-1}(A^{-1}-C^{-1})\\
	&=[C^{-1}(C-B)B^{-1}]^{-1}[C^{-1}(C-A)A^{-1}]  \\
  &=[B(C-B)^{-1}C][C^{-1}(C-A)A^{-1}]    \\
	&=B(C-B)^{-1}[CC^{-1}](C-A)A^{-1}    \\
	&=B(C-B)^{-1}[CC^{-1}](C-A)A^{-1}    \\
	&=B(C-B)^{-1}[I](C-A)A^{-1}    \\
	&=B(C-B)^{-1}(C-A)A^{-1}    \\
	&=B[(C-B)^{-1}(C-A)]A^{-1}    \\
	&=B[([-I](B-C))^{-1}[-I](A-C)]A^{-1}    \\
	&=B[(B-C)^{-1}[-I]^{-1}[-I](A-C)]A^{-1}    \\
	&=B[(B-C)^{-1}[-I][-I](A-C)]A^{-1}    \\
	&=B[(B-C)^{-1}(A-C)]A^{-1}    \\
	&=B[r(B,A;C)]A^{-1}. 
\end{aligned}
\end{equation*}
From skew fields properties we have that $(abc)^{-1}=c^{-1}b^{-1}a^{-1}$and $ab\neq ba$,  $[-I]^{-1}=-I$, and $[-I][-I]=I.$
\qed

\begin{corollary}
In the Pappus affine plane, for three point different from point $O$, in $\ell^{OI}-$line, we have
\[ r(A^{-1},B^{-1};C^{-1})=r(A,B;C) \cdot r(B,A;O).\] 
\end{corollary}
\proof In theorem \ref{ratio-inverses} we prove that,
\[ r(A^{-1},B^{-1};C^{-1})=B[r(A,B;C)]A^{-1}.\] 
We also know that the skew field, which constructed on a line in Pappus affine plane, is commutative, so have
\[ r(A^{-1},B^{-1};C^{-1})=B[r(A,B;C)]A^{-1}=r(A,B;C)[A^{-1}B].\] 
We also know that $A=A-O$ and $B=B-O$, from construction of skew fields, then,
\[A^{-1}B=(A-O)^{-1}(B-O)=r(B,A;O).
\]
Hence,
\[ r(A^{-1},B^{-1};C^{-1})=r(A,B;C) r(B,A;O).\] 
\qed

We can say, now, that all three points in $\ell^{OI}-$line produce a point, which we have called 'ratio'. So, we can construct the map:
\[
r_{BC}(X)=r(X,B;C)=(B-C)^{-1}(X-C), \forall X \in \ell^{OI}
\]
We call this map, '\textbf{equation of line which passes through the points $B$ and $C$}', (in fact, this is the line $\ell^{OI}$, since $B,C \in \ell^{ OI}$).

For this line we have, $r_{BC}(C)=r(C,B;C)=(B-C)^{-1}(C-C)=O$ and 
$r_{BC}(B)=r(B,B;C)=(B-C)^{-1}(B-C)=I$, and for all point $X \in \ell^{OI}$ have $r_{BC}(X)=r(X,B;C)$, which is the ratio point.

If we write the equation of the line $\ell^{OI}$, in the classical form, as $M\cdot X+N$, for two fixed points $M,N \in \ell^{OI}$ and $ \forall X \in \ell^{OI}$. 
In fact, it is easy to find who the points $M$ and $N$ are, we transform the expression,
\[
r_{BC}(X)=r(X,B;C)=(B-C)^{-1}(X-C)=(B-C)^{-1}\cdot X+ [(B-C)^{-1}(-C)]
\]
so have points $M=(B-C)^{-1}$ and $N=(B-C)^{-1}(-C)$.

And since for the three points $A,B,C \in \ell^{OI}$ we have that 'lines' $\ell^{CA}\equiv \ell^{CB}\equiv \ell^{OI}$, we have that,
\[
r_{CA}(X)=r(X,C;A)=(C-A)^{-1}(X-A), \forall X \in \ell^{OI}
\]
and
\[
r_{CB}(X)=r(X,C;B)=(C-B)^{-1}(X-B), \forall X \in \ell^{OI}
\]
also
\[ r_{CA}(X)=M\cdot X +N \quad \text{and} \quad r_{CB}(X)=M\cdot X +N.
\]
These equations are true and for points, $B$ and $A$, respectively, so
\[
r_{CA}(B)=r(B,C;A)=(C-A)^{-1}(B-A)=M\cdot B +N,
\]
and
\[
r_{CB}(A)=r(A,C;B)=(C-B)^{-1}(A-B)=M\cdot A+N.
\]
If we subtract side-by-side, we have,
\[
r_{CB}(A)-r_{CA}(B)=((C-B)^{-1}(A-B) - (C-A)^{-1}(B-A)=MA+N-MB-N,
\]
so
\[
(C-B)^{-1}(A-B)-(C-A)^{-1}(B-A)=M(A-B)
\]
thus
\[
(C-B)^{-1}(A-B)+(C-A)^{-1}(A-B)=M(A-B)
\]
since $B-A=[-I](A-B)$ and 
\[ (C-A)^{-1}(B-A)=(C-A)^{-1}[-I](A-B)=[-I](C-A)^{-1}(A-B). \]
The factor $A-B$ appears on both sides and we factorize this and have,
\[
[(C-B)^{-1}+(C-A)^{-1}](A-B)=M(A-B)
\]
Since the points $A$ and $B$, we have assumed them to be different, we have that $A-B\neq O$, then, there exists the inverse $(A-B)^{-1}$. We multiply from the right side by side by $(A-B)^{-1}$ and we have,
\[
[(C-B)^{-1}+(C-A)^{-1}](A-B)(A-B)^{-1}=M(A-B)(A-B)^{-1},
\]
hence,
\[
[(C-B)^{-1}+(C-A)^{-1}]=M.
\]
Now we can discuss a little about this value of point $M$, and ask whether the point $M$ can be exactly the point $O$.

If $M=O$, we have that,
\[ (C-B)^{-1}+(C-A)^{-1}=O \Leftrightarrow (C-B)^{-1}=-(C-A)^{-1} \Leftrightarrow (C-B)=-(C-A) \]
This is equivalent with,
\[ C+C=A+B \quad \text{or}\quad 2C=A+B. \]

If characteristic of skew field $K=(\ell^{OI}, +, \cdot)$ is different from $2$ ($Char(K)\neq 2$), we have that, the point $C$ is \emph{"midpoint"} of points $A$ and $B$.

If characteristic of skew field $K=(\ell^{OI}, +, \cdot)$ is $2$, ($Char(K)=2$) then we have that, $A+B=O$ and $2A=O$, so $A+B=2A \Leftrightarrow A-B=O$, for this we have $A=B$ which is against our assumption!

Therefore, the point $M=O$, if and only if characteristic of skew field $K=(\ell^{OI}, +, \cdot)$ is different from $2$, and if $A+B=C+C.$

For $M=O$, we have true the equation
\[
r_{CB}(A)=r(A,C;B)=(C-B)^{-1}(A-B)=N, \quad\text{and} \quad A+B=C+C,
\]
We substitute $A=C+C-B$ and have,
\[
\begin{aligned}
N&=(C-B)^{-1}(C+C-B-B)\\
&=(C-B)^{-1}[(C-B)+(C-B)]\\
&=(C-B)^{-1}(C-B)+(C-B)^{-1}(C-B)\\
&=I+I=2I.
\end{aligned}
\]
So, $N=I+I$. We also see the other equation,

\[
r_{CA}(B)=r(B,C;A)=(C-A)^{-1}(B-A)=N, \quad\text{and} \quad A+B=C+C,
\]
We substitute $B=C+C-A$ and have,
\[
\begin{aligned}
N&=(C-A)^{-1}(B-A)\\
&=(C-A)^{-1}(C+C-A-A)\\
&=(C-A)^{-1}[(C-A)+(C-A)]\\
&=(C-A)^{-1}(C-A)+(C-A)^{-1}(C-A)\\
&=I+I.
\end{aligned}
\]
So, $N=I+I$.

In general, the form $MX+N$ is line equation for the line $\ell^{OI}$, we can find exactly, who are the points $M$ and $N$, in $\ell^{OI}-$line. Since we know that the points $O,I\in \ell^{OI}$.

Indeed, let us have, points $B,C \ell^{OI}$, we have the line equation for $\ell^{BC}\equiv\ell^{OI}$,
\[
r_{BC}(X)=r(X,B;C)=(B-C)^{-1}(X-C), \forall X \in \ell^{OI}
\]
and on the other hand, we will have,
\[
r_{BC}(X)=MX+N, \forall X \in \ell^{OI}
\]
so,
\[
(B-C)^{-1}(X-C)=MX+N, \forall X \in \ell^{OI}
\]
Since $O\in \ell^{OI}$, we have true the equations,
\[
(B-C)^{-1}(O-C)=M\cdot O+N=N
\]
so the point $N$, is,
\[
N=-(B-C)^{-1}C \quad \text{or}\quad N=(C-B)^{-1}C
\]

Since $I\in \ell^{OI}$, we have true the equations,
\[
(B-C)^{-1}(I-C)=M\cdot I+N,
\]
replace $N$, found above, and we have
\[
(B-C)^{-1}(I-C)=M+N\Leftrightarrow M=(B-C)^{-1}(I-C)-[-(B-C)^{-1}C]
\]
So
\[
M=(B-C)^{-1}(I-C)+(B-C)^{-1}C=(B-C)^{-1}[(I-C)+C]
\]
Hence the point $M$, is,
\[
M=(B-C)^{-1}.
\]

So for two different-fixed point $B,C\in \ell^{OI}-$line we have the map,
\[
r_{BC}: \ell^{OI} \to \ell^{OI}, 
\]
such that,
\[
\forall X \in \ell^{OI}, r_{BC}(X)=(B-C)^{-1}(X-C).
\]
\begin{theorem}
This ratio-map, $r_{BC}: \ell^{OI} \to \ell^{OI}$ is a bijection in $\ell^{OI}-$line in Desargues affine plane. 
\end{theorem}
\proof
\textbf{Injective:} If, have $r_{BC}(X)=r_{BC}(Y)$, then $X=Y$, truly,
\[r_{BC}(X)=r_{BC}(Y) \Leftrightarrow (B-C)^{-1}(X-C)=(B-C)^{-1}(Y-C)
\]
so
\[ \begin{aligned}
(B-C)^{-1}(X-C)-(B-C)^{-1}(Y-C) &=O \\
&\Leftrightarrow (B-C)^{-1}[(X-C)-(Y-C)] &=O\\
& \Leftrightarrow (B-C)^{-1}(X-Y) &=O
\end{aligned}\]
the above factors are points of the line $\ell^{OI}$, so they are also elements of skew fields $(\ell^{OI}, +, \cdot)$, and we keep in mind that skew field has no divisor of zero, therefore 

\[(B-C)^{-1}(X-Y)=O \Leftrightarrow X-Y=O \Leftrightarrow X=Y.
\]

\textbf{Syrjective:} $\forall R \in \ell^{OI},  \exists X\in \ell^{OI}$, such that, $R=r_{BC}(X)$.

From definition of ratio of three points, we have
\[r_{BC}(X)=(B-C)^{-1}(X-C)=R \Rightarrow X-C=(B-C)R
\]
so
\[ X=(B-C)R+C. \]
Hence, $\forall R \in \ell^{OI},  \exists X=(B-C)R+C \in \ell^{OI}$, such that, $R=r_{BC}[(B-C)R+C]$.

Points $B,C$ are fixed and different points in $\ell^{OI}$, and the point $R$ is in $\ell^{OI}$, we can construct in the line the points $B-C$, $(B-C)R$ and $(B-C)R+C$, and have that,
\[ \begin{aligned}
r_{BC}[(B-C)R+C]&=(B-C)^{-1}[(B-C)R+C-C]\\
&=(B-C)^{-1}[(B-C)R] \\
&=[(B-C)^{-1}(B-C)]R \\
&=I\cdot R\\
&=R.
\end{aligned}
\]
\qed

We mark the set of maps-ratio-three-points, with,

\[
\mathcal{R}_3=\{r_{BC}(X)|\forall X\in \ell^{OI} \}
\]
It's clear that $\mathcal{R}_3 \subseteq \ell^{OI},$ and $\mathcal{R}_3 \neq \emptyset$, below we will prove that this set, is a sub-skew field of $K=(\ell^{OI},+, \cdot )$.  Let's see now the following theorems,
\begin{theorem}\label{thm.11}
Let's have two differetn-fixet points $B, C \in \ell^{OI}$  and this points are different for point $O$, in  $\ell^{OI}-$line. The ratio-maps-set $\mathcal{R}_3=\{r_{BC}(X)|\forall X\in \ell^{OI} \}$ forms a \emph{Abelian} (commutative) group with addition of points in $\ell^{OI}-$line.
\end{theorem}
\proof
\textbf{1.} (Associativity) $\forall r_{BC}(X),r_{BC}(Y),r_{BC}(Z) \in \mathcal{R}_3$, we have three points $ X,Y,Z \in \ell^{OI}$, and, we have that:
\[
\begin{aligned}
\left[r_{BC}(X)+r_{BC}(Y)\right]+r_{BC}(Z)&=[(B-C)^{-1}(X-C)+(B-C)^{-1}(Y-C)]+(B-C)^{-1}(Z-C) \\
&= (B-C)^{-1}\{[(X-C)+(Y-C)]+(Z-C) \} \\
&\text{(from asociatiove properties of addition groups $(\ell^{OI},+)$)}\\
&= (B-C)^{-1}\{(X-C)+[(Y-C)+(Z-C)] \} \\
&=(B-C)^{-1}(X-C)+(B-C)^{-1}[(Y-C)+(Z-C)]   \\
&=(B-C)^{-1}(X-C)+[(B-C)^{-1}(Y-C)+(B-C)^{-1}(Z-C)]   \\
&=r_{BC}(X)+[r_{BC}(Y)+r_{BC}(Z)],
\end{aligned}
\]
so,
\[
[r_{BC}(X)+r_{BC}(Y)]+r_{BC}(Z)=r_{BC}(X)+[r_{BC}(Y)+r_{BC}(Z)], \forall X,Y,Z \in \ell^{OI}.
\]
\textbf{2.} (Commuativity) $\forall X,Y \in \ell^{OI}$, we have that:
\[
\begin{aligned}
r_{BC}(X)+r_{BC}(Y)&=(B-C)^{-1}(X-C)+(B-C)^{-1}(Y-C)\\
&= (B-C)^{-1}[(X-C)+(Y-C)] \\
&\text{(from commutative property of addition groups $(\ell^{OI},+)$)}\\
&= (B-C)^{-1}[(Y-C)+(X-C)] \\
&=(B-C)^{-1}(Y-C)+(B-C)^{-1}(X-C) \\
&=r_{BC}(Y)+r_{BC}(X),
\end{aligned}
\]
so
\[  r_{BC}(X)+r_{BC}(Y)=r_{BC}(Y)+r_{BC}(X),  \forall X,Y \in \ell^{OI}\]

\textbf{3.} ('zero') $\exists Z \in \ell^{OI}, \forall X \in \ell^{OI}$, we have that:
\[
\begin{aligned}
r_{BC}(X)+r_{BC}(Z)&=r_{BC}(X)\\
(B-C)^{-1}(X-C)+(B-C)^{-1}(Z-C) &=(B-C)^{-1}(X-C)\\
(B-C)^{-1}[(X-C)+(Z-C)] &=B^{-1}(X-C)
\end{aligned}
\]
since $B\neq C$ then, $B-C \neq O$ then exists $(B-C)^{-1}$ and is different from point $O$. So we have
\[ 
\begin{aligned}
(B-C)^{-1}[(X-C)+(Z-C)] &=B^{-1}(X-C) \\
 (B-C)^{-1}[(X-C)+(Z-C)]-B^{-1}(X-C) &=O 
\end{aligned}
\]
so,
\[ (B-C)^{-1}[(X-C)+(Z-C)-(X-C)]=O \Leftrightarrow (B-C)^{-1}(Z-C)=O, \]
we also know that these are points of skew-fields $(\ell^{OI},+,\cdot )$, and skew fields do not have 'divisors of zero', therefore,
\[(B-C)^{-1}(Z-C)=O \Leftrightarrow Z-C=O\Leftrightarrow Z=C.
\]
Hence the 'zero' element is $r_{BC}(C)$.

\textbf{4.} ('opposite') $\forall X \in \ell^{OI}, \exists \bar{X} \in \ell^{OI}$, we have that:
\[
\begin{aligned}
r_{BC}(X)+r_{BC}(\bar{X})&=r_{BC}(C)\\
(B-C)^{-1}(X-C)+(B-C)^{-1}(\bar{X}-C)&=(B-C)^{-1}(C-C)\\
(B-C)^{-1}[(X-C)+(\bar{X}-C)]&=O\\
(B-C)^{-1}[X-C+\bar{X}-C]&=O \\
(B-C)^{-1}[X+\bar{X}-2C]&=O \\
\end{aligned}
\]
so
\[X+\bar{X}-2C=O \Rightarrow \bar{X}=2C-X
\]

Hence the 'opposite' element of $r_{BC}(X)$ is $r_{B}(2C-X)$, $\forall X \in \ell^{OI}$.

\qed


\begin{theorem} \label{thm.12}
Let's have tow different points $B,C$ and different for point $O$, in  $\ell^{OI}-$line. The ratio-maps-set $\{r_{BC}(X)|\forall X\in \ell^{OI} \}$ forms a group with 'multiplication of points in $\ell^{OI}-$line.
\end{theorem}
\proof
\textbf{1.} (Associativity) $\forall X,Y,Z \in \ell^{OI}$, we have that:
\[
\begin{aligned}
\left[r_{BC}(X) \cdot r_{BC}(Y)\right]\cdot r_{BC}(Z)&=[(B-C)^{-1}(X-C) \cdot (B-C)^{-1}(Y-C)]\cdot (B-C)^{-1}(Z-C) \\
&\text{all the factors are points of the line $\ell^{OI}$, }\\
&\text{they are also elements of skew fields, therefore from the} \\
&\text{association property, we can move the brackets}\\
&=(B-C)^{-1}(X-C) \cdot [(B-C)^{-1}(Y-C)\cdot (B-C)^{-1}(Z-C) ]\\
&=r_{BC}(X) \cdot [r_{BC}(Y) \cdot r_{BC}(Z)]
\end{aligned}
\]

\textbf{2.} ('unitary') $\exists E \in \ell^{OI}, \forall X \in \ell^{OI}$, we have that:
\[
\begin{aligned}
r_{BC}(X)\cdot r_{BC}(E)&=r_{BC}(X)\\
(B-C)^{-1}(X-C) \cdot (B-C)^{-1}(E-C)&=(B-C)^{-1}(X-C)\\
(X-C) \cdot (B-C)^{-1}(E-C)&=(X-C) \\
(X-C)^{-1}(X-C) \cdot (B-C)^{-1}(E-C)&=(X-C)^{-1}(X-C) \\
(B-C)^{-1}(E-C)&=I \\
(E-C)&=I \cdot (B-C) \\
E-C&=B-C \\
E&=B
\end{aligned}
\]
so unitary-element element is $r_{BC}(B)$.

\textbf{3.} ('inverse') $\forall X \in \ell^{OI}, \exists Y \in \ell^{OI}$, we have that:
\[
\begin{aligned}
r_{BC}(X)\cdot r_{BC}(Y)&=r_{BC}(B)\\
(B-C)^{-1}(X-C) \cdot (B-C)^{-1}(Y-C)&=(B-C)^{-1}(B-C)\\
 (X-C) \cdot (B-C)^{-1}(Y-C)&=B-C\\
(B-C)^{-1}(Y-C)&=(X-C)^{-1}\cdot (B-C)\\
r_{BC}(Y)&=r_{XC}(B)\\
&\text{or, distinguishing $Y$, we have}\\
Y-C&=(B-C)\cdot (X-C)^{-1}\cdot (B-C)\\
Y&=(B-C)\cdot (X-C)^{-1}\cdot (B-C) +C
\end{aligned}
\]
Hence the 'inverse' element of $r_{BC}(X)$ is $r_{XC}(B)$, $\forall X \in \ell^{OI}$.
\qed

\begin{theorem} \label{thm.13}
In the ratio-maps-set $\{r_{BC}(X)|\forall X\in \ell^{OI} \}$, for three elements $r_{BC}(X)$, $r_{BC}(Y)$, $r_{BC}(Z)$ of its, have true, that equations
\begin{enumerate}
	\item $r_{BC}(X)\cdot [r_{BC}(Y)+r_{BC}(Z)]=r_{BC}(X)\cdot r_{BC}(Y) +r_{BC}(X)\cdot r_{BC}(Z)$, and,
	\item $[r_{BC}(X)+r_{BC}(Y)]r_{BC}(Z)=r_{BC}(X)\cdot r_{BC}(Z) +r_{BC}(Y)\cdot r_{BC}(Z)$
\end{enumerate}
\end{theorem}
\proof Indeed 
\[
\begin{aligned}
r_{BC}(X) \cdot [r_{BC}(Y)+r_{BC}(Z)]&=(B-C)^{-1}(X-C) [ (B-C)^{-1}(Y-C) + (B-C)^{-1}(Z-C)] \\
&\text{all the factors are points of the line $\ell^{OI}$, }\\
&\text{they are also elements of skew fields,}\\
&\text{therefore from the distribution property, have} \\
&=(B-C)^{-1}(X-C) \cdot (B-C)^{-1}(Y-C) \\
&+ (B-C)^{-1}(X-C) \cdot (B-C)^{-1}(Z-C) \\
&=r_{BC}(X) \cdot r_{BC}(Y) + r_{BC}(X)\cdot r_{BC}(Z).
\end{aligned}
\]
Hence,  $ r_{BC}(X)\cdot [r_{BC}(Y)+r_{BC}(Z)]=r_{BC}(X)\cdot r_{BC}(Y) +r_{BC}(X)\cdot r_{BC}(Z).$

The other point (2) is proved in the same way, so have
\[
\begin{aligned}
\left[r_{BC}(X) +r_{BC}(Y)\right]\cdot r_{BC}(Z)]&=[(B-C)^{-1}(X-C) + (B-C)^{-1}(Y-C)] \cdot (B-C)^{-1}(Z-C)] \\
&\text{all the factors are points of the line $\ell^{OI}$, }\\
&\text{they are also elements of skew fields,}\\
&\text{therefore from the distribution property, have} \\
&=(B-C)^{-1}(X-C) \cdot (B-C)^{-1}(Z-C) \\
&+ (B-C)^{-1}(Y-C) \cdot (B-C)^{-1}(Z-C) \\
&=r_{BC}(X) \cdot r_{BC}(Z) + r_{BC}(Y)\cdot r_{BC}(Z).
\end{aligned}
\]
Hence,  $r_{BC}(X) +r_{BC}(Y)]\cdot r_{BC}(Z)]=r_{BC}(X) \cdot r_{BC}(Z) + r_{BC}(Y)\cdot r_{BC}(Z).$

\qed

From the above three theorems \ref{thm.11}, \ref{thm.12} and \ref{thm.13}, have true, this, 
\begin{theorem}
The ratio-maps-set $\mathcal{R}_3=\{r_{BC}(X)|\forall X\in \ell^{OI} \}$, for a different fixed points $B,C$ in $\ell^{OI}-$line, forms a skew-field with 'addition and multiplication' of points in $\ell^{OI}-$line.

This, skew field $(\mathcal{R}_3, +, \cdot)$ is sub-skew field of the skew field $(\ell^{OI}, +, \cdot)$.
\end{theorem}

\section{Free Group Presentation of Affine Plane Cycles}
This section briefly introduces free groups resulting from cycles that share colinear ratio points in simply-connected polygons in the Desargues affine plane. Polygons are {\bf simply connected}, provided the polygon vertices are in a sequence of edges with no self-intersections.  

Free groups were introduced by W. Dyck in 1882 as a natural biproduct of simply connected polygons $P$ in which every vertex $a\in P$ can be reached by traversing the edges between $a$ and a distinguished vertex $g\in P$~\cite{Dyck1882freiGruppe}.

\begin{definition}\label{def:DyckPolygon}{\rm \bf Dyck Polygon}.\\
A {\bf Dyck polygon} is a simply-connected polygon in the Desargues affine plane.
\qquad \textcolor{black}{\Squaresteel}
\end{definition}

\begin{example}\label{ex:3Triangles}
A simply-connected polygon that consists of three overlapping triangles with colinear ratio points $A,B,C$ in common is shown in Fig.~\ref{fig:3cycles}.  This polygon contains cycles $\cyc E_1, \cyc E_2, \cyc E_3$:
\begin{compactenum}[1$^o$]
\item $\cyc E_1 = A\to A-C\to B_3\to A$.
\item $\cyc E_2 = B\to B-C\to B_4\to B$.
\item $\cyc E_3 = C\to O\to B_1\to B$.\qquad \textcolor{black}{\Squaresteel}
\end{compactenum}
\end{example}
\begin{figure}[!ht]
\centering
\includegraphics[width=0.90\textwidth]{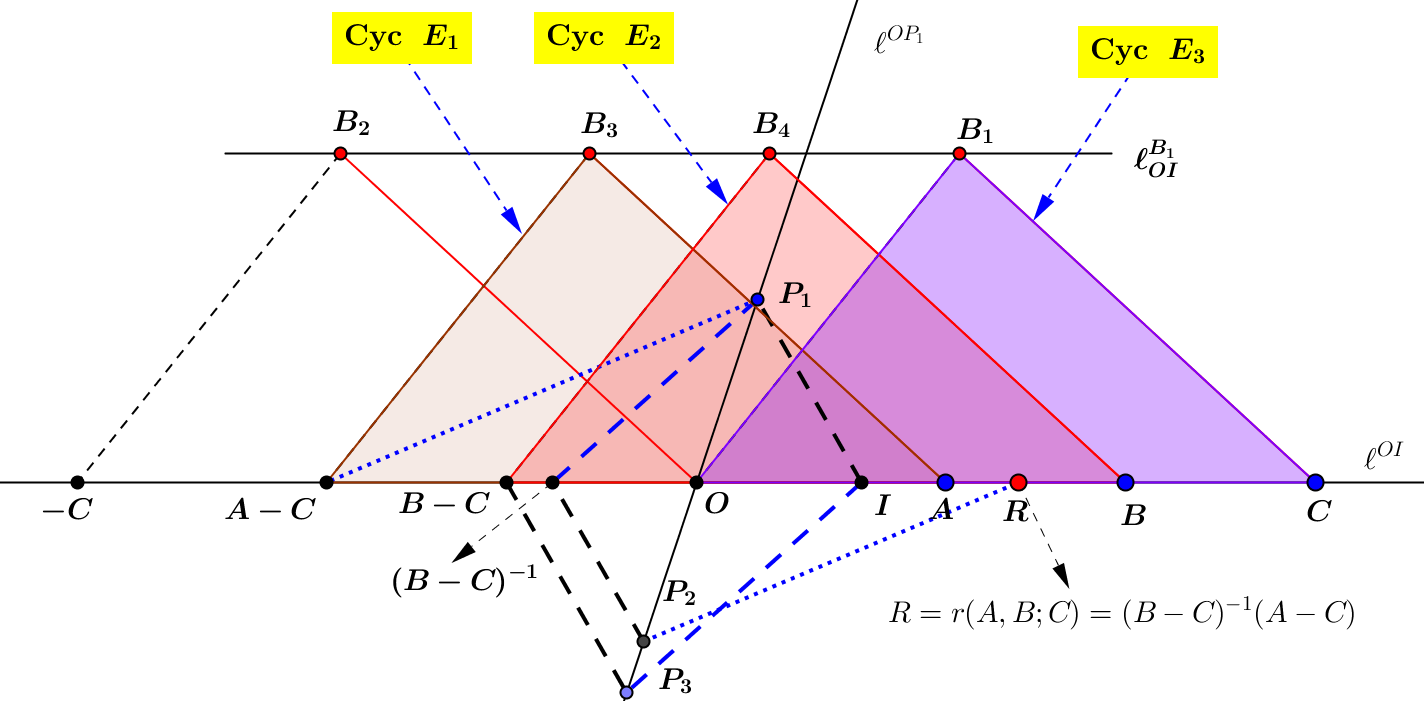}
\caption[]{Polygon Cycles $\left\{\mathop{\bigcup}\limits_{1\leq i \leq 3}\mbox{\bf cyc}E_i\right\}\to\ \mbox{\bf free group}\ G(\beta,+)$}
\label{fig:3cycles}
\end{figure}

Notice that it is possible to start at any one of the three ratio points $A,B,C$ and traverse a sequence of edges to reach any other vertex in the polygon, {\em e.g.}, if we start at $A$, we have
\begin{align*}
A &\to B \to B_4.\ \mbox{or},\\
A &\to B \to C\to B_1.\ \mbox{or},\\
A &\to O\to B-C.\ \mbox{or},\\
A &\to O\to B-C\to A-C.
\end{align*}
From Example~\ref{ex:3Triangles}, we see that the ratio points $A,B,C$ in the polygon (3 overlapping triangles) serve as distinguished vertices (called generators) so that any other vertex in the polygon can be reached from a generator vertex by traversing a sequence of edges. 
\begin{definition}
A {\em free group} $G(\beta,+)$ is nonempty finite set $G$ of $n$ elements equipped with binary operation (typically written as a $+$) and a basis $\beta\subset G$ so that every member $a\in G$ can be written as a linear combination the generators $g\in \beta$, {\em i.e.}, $a =\mathop{\sum}\limits_{g\in\beta} kg, k\ \mbox{mod}\ n.$
\qquad \textcolor{black}{\Squaresteel}
\end{definition}
\begin{lemma}\label{lemma:DyckPolygon}
Every vertex in a Dyck polygon containing colinear ratio points can be reached by a finite sequence of moves from a ratio point.
\end{lemma}
\begin{proof}
Let $P$ be a Dyck polgon.  From Def.~\ref{def:DyckPolygon}, $P$ is simply-connected.  Hence, if vertices $g,v\in P$, then there is a finite sequence of $k$ moves to reach $v$ from $g$, denoted by $v = kg$. 
\end{proof}
\begin{theorem}\label{theorem:DyckSum}
Every vertex in a Dyck polygon containing colinear ratio points is a linear combination of the ratio points.
\end{theorem}
\begin{proof}
Let $P$ be a Dyck polgon containing colinear ratio points $A,B,C$ and let vertex $v\in P$.  Let $k_1,k_2,k_3\in \mathbb{Z}$. From Lemma~\ref{lemma:DyckPolygon}, we have
\begin{align*}
B &= k_1A.\\
C &= k_2B + k_1A.\\
v &= k_3C + k_2B + k_1A\\
  &= \sum_{\substack{k_i\in \mathbb{Z}\\ 
	g_j\in P}} k_i g_j.
\end{align*}
\end{proof}

\begin{definition}\label{def+}{\rm {\bf Free group move operation +}}.\\
A Dyck free group denoted by $G(\beta,+)$ is defined in terms of a seqeunce of moves from one vertex to another one in a polygon $G$ and basis $\beta$ containing collinear ratio points.  From Theorem~\ref{theorem:DyckSum}, every vertex in $G$ is a linear combination of the elements of the basis $\beta$.
For a Dyck free group $G(\beta,+)$, the + operation is defined by a continuous map + such that
\begin{align*}
\mathbb{Z} &= \mbox{integer}.\\
X &= \mbox{set of vertices in a polygon}.\\
v,v' &\in X.\\
G &= \left\{g\in X\right\}.\\
\beta &= \left\{g\in P\right\}.\\
g &\in \beta.\\
+:\mathbb{Z}\times X\times \mathbb{Z}\times X&\to \mathbb{Z}\times X,\ \mbox{defined by}\\
+(x,x') &= x + x'\\ 
        &= \sum_{\substack{k_i\in \mathbb{Z}\\ 
				g_j\in \beta}} 
			  k_i g_j\ \mbox{from Theorem~\ref{theorem:DyckSum}}.\\
+(v,0v') &= \overbrace{kg+0g' = kg + 0 = v\in X.}^{\mbox{\textcolor{black}{\bf zero moves from $\boldsymbol{v}$$\in X$}}}\\
+(v,-v) &= \overbrace{kg-kg = (k-k)g = 0g = g+0 = v\in X.}^{\mbox{\textcolor{black}{\bf move back to $\boldsymbol{v}$ from $\boldsymbol{v}$}}}	
\mbox{\qquad \textcolor{black!20}{\Squaresteel}}
\end{align*}
\end{definition}

\begin{lemma}\label{lemma:freeGroupoid}
Let $G(\beta,+)$ be a Dyck polygon $G$ with basis $\beta$ and equipped with move +.  The structure $G(\beta,+)$ is a free groupoid.
\end{lemma}
\begin{proof}
From Theorem~\ref{theorem:DyckSum}, $G$ has a basis $\beta\subset G$. Let $g\in \beta$.
From Def.~\ref{def+}, + is a binary operation, {\em i.e.}, every $+(v,v') = kg + k'g\in G$ for $v,v'\in G$.  Hence, $G(\beta,+)$ is a free groupoid.
\end{proof}

\begin{lemma}\label{lemma:DyckSemigroup}
Let $G(\beta,+)$ be a free groupoid.
Then $G(\beta,+)$ is a free semigroup.
\end{lemma}
\begin{proof}
From Lemma~\ref{lemma:freeGroupoid}, the structure $G(\beta,+)$ is a free groupoid on a Dyck polygon $G$ with basis $\beta$ and operation $+$.  Let $+(v,v') = kg + k'g$ for $v,v'\in G, g\in \beta$.  Then
\begin{align*}
v=kg,v'=k'g,v''=kk''g&\ \mbox{for}\ v,v',v''\in G,g\in\beta\\
+((v,v'),v'') &= (kg + k'g) + k''g\\
 &= kg + (k'g + k''g)\\
 &= (v,(v',v'')).
\end{align*}
Consequently, the operation + is associative.  Hence, $G(\beta,+)$ is a Dyck free semigroup.
\end{proof}

\begin{lemma}\label{lemma:identity} Every Dyck free semigroup has an identity element.
\end{lemma}
\begin{proof}
Immediate from Def.~\ref{def+}, since we can always write
\[
kg + 0g = (k+0)g = v\in G,
\]
for every vertex $v\in G$.
\end{proof} 

\begin{notation}{\rm {\bf Psudo-Identity element}}.\\
The psudo-identity element 0 in a Dyck free semigroup stems from integer 0, since we can always write
\[
0g + kg = 0 + kg = kg = v\in G, 0,k\in \mathbb{Z}.
\]
We also write
\[
0g = 0 + 1g = v\in G.
\]
The notation $0 + 1g$ reads {\bf traversal from} $v$ with zero length.
\qquad \textcolor{black!20}{\Squaresteel}
\end{notation}

\begin{lemma}\label{lemma:inverse} Every member of a Dyck free semigroup has an inverse.
\end{lemma}
\begin{proof}
Immediate from Def.~\ref{def+}, since we can always write
\[
kg - kg = (k-k)g = 0v = v\in G
\]
for every vertex $v\in G$.  
\end{proof}

\begin{theorem}\label{theorem:freeDyckGroup}
$G(\beta,+)$ is a free Dyck group.
\end{theorem}
\begin{proof}
From Lemma~\ref{lemma:DyckSemigroup}, the structure $G(\beta,+)$ is a free semigroup with basis $\beta$.  From Lemma~\ref{lemma:identity}, $G$ has an identity element, namely, $0v = v$.  From Lemma~\ref{lemma:inverse}, every member $v\in G$ has inverse, namely, $kg - kg = (k-k)g = 0g = v$.  Hence,
$G(\beta,+)$ is a free Dyck group.
\end{proof}

\begin{theorem}\label{theorem:DyckPolgonPresentation}
Every Dyck polygon containing colinear ratio vertices has a free group presentation.
\end{theorem}
\begin{proof}
Immediate from Theorem~\ref{theorem:freeDyckGroup}.
\end{proof}

From a Dyck free group presentation of a Desargues affine polygon containing colinear, we obtain a concise means of identify the properties of colinear ratio vertices, {\em e.g.},
\begin{compactenum}[{\bf Property-}1$^o$]
\item {\bf Generator} Colinear ratio vertices are generators in a Dyck free group.
\item {\bf Traversal} In a Dyck free group on a Dyck polygon, there is a traversal from any ratio vertex to any other vertex in the polygon.
\item {\bf Path Measure}.  Every vertex in a Dyck polygon with colinear ratio vertices can be written concisely as a linear combination of the ratio vertices.  Hence, path has length $k_1+\cdots+k_n$ for $v = \sum_{\substack{k_i\in \mathbb{Z}\\ 
				g_j\in \beta}} 
			  k_i g_j'$.
\item {\bf Concise Presentation} From Theorem~\ref{theorem:DyckPolgonPresentation}, every Dyck polygon with colinear ratio points has a free group presentation.  
\item {\bf Path}  From an application perspective, Theorem~\ref{theorem:DyckPolgonPresentation} is significant, since we can then view a discrete Feynman path~\cite[p. xiv]{Feynman1942thesis} (a trace of a trajectory of a particle) having a geometric realization as a sequence of edges in a  Dyck polygon containing colinear ratio points.
\end{compactenum}
%
%
%
\bibliographystyle{amsplain}
\bibliography{RCRrefs}

\end{document}